\documentclass{siamltex}
\usepackage{amsmath}
  \usepackage{epsfig,graphicx,psfrag,mathrsfs}   
\usepackage{amsmath,amssymb,dsfont,upgreek,multirow}
\usepackage{hyperref,color,url}  
\usepackage[algoruled, algonl]{algorithm2e}

\usepackage{pifont}

\newcommand{\n}{\ding{55}}

\newtheorem{example}{Example}[section]

\newtheorem{Lemma}{Lemma}[section]
\newtheorem{remark}{Remark}[section]
\title{Method of Successive Projections for Nonnegative Inverse Singular Value problems with Prescribed Structure
\\
}

\author{
Sheng-Jhih Wu\thanks{Institute of Mathematics, Academia Sinica, 6F, Astronomy-Mathematics Building, No.1, Sec. 4, Roosevelt Road, Taipei 10617, Taiwan. (sjw@math.sinica.edu.tw).}
\and Matthew M. Lin \thanks{Corresponding author. Department of Mathematics, National Chung Cheng University, Chia-Yi 621, Taiwan. (mhlin@ccu.edu.tw).
This
research was supported in part by the National Science Council of Taiwan under grant101-2115-M-194-007-MY3}
}

\begin{document}

\maketitle

\begin{abstract}
The study of solving  inverse singular value problems for nonnegative matrices has been around for decades. 
It is clear that an inverse singular problem is trivial if the desirable matrix is not restricted to a certain structure. Provided with singular values
and diagonal entries, this paper presents a numerical procedure, based on the successive projection process, to solve inverse singular value problems  for nonnegative matrices subject to given diagonal entries.
Even if we focus on the specific type of inverse singular value problems
with prescribed diagonal entries, this entire procedure can be carried over with little effort to other types of structure.   Numerical examples are used to demonstrate the capacity and efficiency of our method.

%
% and, more generally, we show that how structure constraints can be added for solving NISVP.



\end{abstract}

\section{Introduction}
For an $n\times n$ complex matrix $A$, the singular values of $A$ 
are the non-negative square roots of the eigenvalues of $A^*A$, 
where the matrix $A^*$ is the conjugate transpose of $A$.
The inverse singular value problems (ISVP) deals with the reconstruction of 
a matrix from certain prescribed  singular values. 
%It is clear that an ISVP becomes clear if a certain structure is not imposed on the desirable matrix. 
In general, the ISVP is a natural extension of the inverse eigenvalue problem (IEP)~\cite{Chu05,Orsi2006} since eigenvalues and singular values are two of the most important, closely related characteristics of a matrix.
Nevertheless, the very intimate relation between these two quantities 
does not mean that the ISVP is a trivial generalization of the IEP.
Though, as noted in \cite{Chu05} and described below, there is a one-to-one correspondence between the solvability
of these two types of inverse problems. That is, 
if $A$ is a matrix given by 
\begin{equation} \label{corresponding solvability}
A = \left[\begin{array}{cc} 0 & B \\  B^{T} & 0
\end{array}\right],
\end{equation}
where $B$ is some structured matrix, then the ISVP for $B$ is solvable if only if  the IEP for $A$ is solvable.
However, solving the ISVP indirectly by the IEP approaches might suffer difficulty due to the extra block diagonal zeros appearing in (\ref{corresponding solvability}). Thus, on how to solve the ISVP directly has received growing attention in the matrix theory community in the recent years.
The reader is referred to \cite{Bai2012,Bai2007a,Bai2008,Chan2003,Chu1992,Chu1999,Chu2000,Ma2012,Montano2008a,Montano2008b,Vong2011,Yuan2012}.
Besides of its own theoretical interest, the ISVP have practical applications 
such as in some quadratic group~\cite{Politi2003}, 
in enforcing passivity of rational macromodels~\cite{Saunders2011},
and in signature sequences design~\cite{Tropp2004} .

Among the various settings for the ISVP in the literature, a scenario worthy of discussion is when a structured matrix is further restricted by prescribed entries. A necessary condition for the existence of a matrix with given diagonal entries $\boldsymbol{d}$ and singular values $\boldsymbol{\sigma}$ is first 
given by de Oliviera~\cite{Oliveira1971}. 
See also Sing~\cite{Sing1976} and Miranda~\cite{Miranda1979} for more general necessary conditions. While considering the necessary and sufficient conditions, a classical result is the Sing-Thompson theorem.
\begin{theorem} \label{thm:ST}
(Sing-Thompson Theorem \cite{Sing1976,Thompson1977}) There exists a real matrix $A \in \mathbb{R}^{n \times n}$
with
singular values $\boldsymbol{\sigma}$ and main diagonal entries $\boldsymbol{d}$, possibly
in different order, if and only if
\begin{equation}
\sum_{i=1}^{k} |d_{i}| \leq \sum_{i=1}^{k}\sigma_{i}
\end{equation}
for all $k=1,2,\ldots ,n$ and
\begin{equation}
\sum_{i=1}^{n-1} |d_{i}| - |d_{n}| \leq \sum_{i=1}^{n-1}\sigma_{i} - \sigma_{n}.
\end{equation}
\end{theorem}

Provided that the given singular values and diagonal elements satisfy the Sing-Thompson conditions, we also refer the reader to \cite{Chu1999}
for a numerical procedure to solve this problem.
Motivated by this, 
our purpose in this investigation is to explore mathematical properties and to formulate a numerical scheme for solving the ISVP with nonnegative entries and prescribed diagonal entries and as a by-product, with prescribed specific structure.  These implementations can both be seen in our subsequent discussion.

To set forth our consideration, we define our major concern as follows

\vskip .1in
\begin{minipage}{0.9\textwidth}
\rm{(}\textbf{NISVP}\rm{)} 
Given a collection of nonnegative real numbers $\boldsymbol{\sigma} = \{ \sigma_{1}, \cdots, \sigma_{n} \}$
and a set of real numbers $\boldsymbol{d}= \{ d_{1}, \cdots, d_{n} \}$,
construct a nonnegative $m \times n$ matrix, where $m \geq n$, with singular values $\boldsymbol{\sigma}$
and principal diagonal entries $\boldsymbol{d}$.

\end{minipage}
\vskip .1in

%It should be emphasize that this problem can be further generalized to reconstruct a nonnegative matrix with prescribed entries subject to a set of restricted locations. This implementation can be seen in our numerical experiment later. In this work, we focus on solving the NISVP 
%with prescribed entries along the diagonal 
%since it is the most natural place which the attention has been centered around in the literature of matrix theory. 

%%
%%
While considering nonnegative properties, it is true that extensive research interests have been done for the study of the spectrum of nonnegative matrices. 
Among the considerably established results, notably the well-known necessary, but not sufficient, condition is the Perron-Frobenius theorem. 
Equipped with fruitful results on nonnegative matrices, naturally, the IEP for nonnegative matrices 
has drawn the attention of researchers.
In contrast to the large amount of study on the nonnegative IEP, the 
nonnegative ISVP has received much less investigation.
To the best of our knowledge, 
there is no literature on conditions of the existence for the nonnegative ISVP with prescribed entries whereas most of the research on the ISVP is devoted to numerical computation only. 
For example, let us state an ISVP, which is originally posed in \cite{Chu1992}.
Given a collection of real general $m \times n$ matrices $A_{0}, A_{1}, \cdots, A_{n}$
where  $m \geq n$ and a sequence of nonnegative real numbers $\sigma_{1}, \cdots, \sigma_{n}$,
find values of real numbers $c := (c_{1}, \cdots, c_{n})^{T} \in \mathds{R}^{n}$ such that
the singular values of the matrix
\begin{equation}
A(c) := A_{0} + c_{1}A_{1} + \cdots + + c_{n}A_{n}
\end{equation}
are exactly $\sigma_{1}, \cdots, \sigma_{n}$.
Various numerical methods have been advanced for this prototype 
including the projected gradient method in \cite{Chu1992},
Newton-type methods in \cite{Chu1992}, \cite{Bai2008}, \cite{Bai2012}, and 
an Ulm-like method in \cite{Vong2011}.
For the ISVP that realizes prescribed diagonal, 
a recursive algorithm is proposed in \cite{Chu1999}.

In this presentation, we apply successive projections technique among certain sets, each with partially desired feature, to solve the NISVP.
Essentially, the procedures are iterative, 
where a singular value decomposition is
the main computation in each iteration.
%The elaborated comments on the notion and formulation 
%of this method will be given in the subsequent section.
The contribution of this paper are two-fold 
with one on the theoretical side and the other on the computational side. 
First, we have not been aware of any study 
on the necessary and sufficient conditions of the existence of the NISVP with prescribed diagonal entries in the literature.
%In fact, this questions for ISVPs are difficult and thus rarely answered.
%A novelty in the present paper is a complete solution to this problem 
%for the nonnegative ISVP matching prescribed diagonal.
Second, a numerical approach is proposed for reconstructing such desired matrices.
From our numerical experiments, 
the algorithms work not only for prescribed diagonal entries, 
but also for arbitrarily assigned entries.
Moreover, with little effort, the numerical procedures developed in the present study should be applicable 
to matrices with other features such as Toeplitz, Hankel, and so on.

The structure of this paper is organized as follows.
We begin in section $2$ with a discussion of how $2\times 2$ matrices can be handled precisely, provided with its singular values and principal diagonal elements.  The problem of finding a general nonnegative matrix, however, remains open.
%
%
%In section $2$, we first offer necessary and sufficient conditions for the existence of
%the $2 \times 2$ nonnegative matrix with prescribed singular values and elements in the diagonal.
%Using the $2 \times 2$ case as a building block,
%the necessary and sufficient conditions for the existence of
%matrices with general size are established.
%%It is a generalization of the result of \cite{Thompson1977} in nature.
In section $3$, the approach of successive projections to solve the NISVP with prescribed diagonal elements is given.  For this approach, a detailed convergence is also given. 
Demonstration of the capacity and efficiency of the numerical approach is presented via numerical examples in section $4$ and the the concluding remark is given in section $5$.

\section{The $2\times 2$ case}
In this section, we demonstrate how a $2 \times 2$ nonnegative matrix 
with assigned singular values and diagonal entries
can be constructed in closed form.
We note that the construction of the $2 \times 2$ case in \cite{Thompson1977}
is in the complex domain. 
The requirement of being real (nonnegativity) considered in the present study 
demands an extra condition, see (\ref{eq:extra}) below,
other than the sufficient and necessary conditions obtained in \cite{Thompson1977}.  
We first offer sufficient condition for the existence of such a matrix in the following Lemma.

\begin{Lemma}\label{lemma:21}
Let $b$, $c$, $\sigma_1\geq\sigma_2$, and $d_1\geq d_2$
be nonnegative numbers. 
Then, there exists a nonnegative matrix 
\begin{equation}\label{eq:A}
A=
\left[\begin{array}{cc}d_1 & b \\c & d_2\end{array}\right]
\in\mathbb{R}^{n\times n}
\end{equation}
having  singular values $\sigma_1\geq\sigma_2$ and principal diagonal entries $d_1\geq d_2$, renumbering if necessary, if   
\begin{subequations}\label{eq:nisvp}
\begin{eqnarray}
  \sigma_1 + \sigma_2 \geq d_1+d_2,\,\,
  \sigma_1 - \sigma_2 &\geq& d_1-d_2,\quad
     \mbox{while }
 bc - d_1d_2 \leq 0,\label{eq:22cond}\\
\sigma_1- \sigma_2 &\geq& d_1 + d_2, \quad  \mbox{while }bc - d_1d_2 > 0.\label{eq:22cond2}
\end{eqnarray}
\end{subequations} 
% \begin{subequations}\label{eq:nisvp}
%\begin{eqnarray}
%  \sigma_1 + \sigma_2 \geq d_1+d_2,\,\,
%  \sigma_1 - \sigma_2 \geq d_1-d_2,
%    & \quad\mbox{as }
% bc - d_1d_2 \leq 0,\label{eq:22cond}\\
%\sigma_1- \sigma_2 \geq d_1 + d_2, &\quad  \mbox{as }bc - d_1d_2 > 0.
%\label{eq:22cond2}
%\end{eqnarray}
%\end{subequations} 

\end{Lemma}

\begin{proof}
To show the existence of a nonnegative 
matrix as it is given in~\eqref{eq:A} 
with singular values $\{\sigma_1, \sigma_2\}$, it suffices to establish that the statement 
\begin{equation*}
\det(\lambda I - A^\top A) = (\lambda - \sigma_1^2)(\lambda - \sigma_2^2) 
\end{equation*}
is true for some entries $b$ and $c$, that is, 
\begin{subequations}\label{eq:22sig}
\begin{eqnarray}
 b^2+c^2 &= & \sigma_1^2+\sigma_2^2-d_1^2-d_2^2,\label{eq:22siga}\\
 bc  &=&  \left\{\begin{array}{cl}
 d_1d_2 - \sigma_1\sigma_2,  & \mbox{if }
 bc - d_1d_2 \leq 0,\\
  \sigma_1\sigma_2+d_1d_2,  & \mbox{if }
 bc-d_1d_2 > 0.
  \\
  \end{array}\right.\label{eq:22sigb}
\end{eqnarray}
\end{subequations}

First, observe that if $ bc - d_1d_2 \leq 0$, then by squaring both sides of~\eqref{eq:22cond}, we have 
\begin{subequations}\label{eq:221cond}
\begin{eqnarray}
  \sigma_1^2 + \sigma_2^2-d_1^2-d_2^2 &\geq& 2(d_1d_2- \sigma_1\sigma_2),\\
   \sigma_1^2 + \sigma_2^2-d_1^2-d_2^2 &\geq& 2( \sigma_1\sigma_2 - d_1d_2),
\end{eqnarray}
\end{subequations}
that is, 
\begin{equation}\label{eq:222cond}
\sigma_1^2 + \sigma_2^2-d_1^2-d_2^2 \geq 2|(d_1d_2- \sigma_1\sigma_2)|\geq 0.
\end{equation}
This implies that the right hand side of~\eqref{eq:22siga} is nonnegative. Also, from~\eqref{eq:222cond}, the solvability of nonnegative solution $(b,c)$% of~\eqref{eq:22sig} 
is ensured, since the vertex of the hyperbola $\mathcal{H}_1$ with equation $bc=|(d_1d_2- \sigma_1\sigma_2)|$ in the positive quadrant lies within the circle $\mathcal{C}$ 
%$b^2+c^2 = \sigma_1^2 + \sigma_2^2-d_1^2-d_2^2$
with center at the origin and radius $(\sigma_1^2 + \sigma_2^2-d_1^2-d_2^2)^{\frac{1}{2}}$.

Second, if $ bc - d_1d_2 > 0$, then a further requirement
\begin{equation*}
( \sigma_1\sigma_2 + d_1d_2 ) \leq (\sigma_1^2 + \sigma_2^2-d_1^2-d_2^2)/2
\end{equation*}
is needed so that the vertex of the hyperbola $\mathcal{H}_2$ with equation  $bc=\sigma_1\sigma_2 + d_1d_2$ in the positive quadrant lies 
within the circle $\mathcal{C}$. It follows that $(b,c)$ is solvable, if~\eqref{eq:22cond2} is satisfied. Since~\eqref{eq:22cond2} implies~\eqref{eq:22cond}, this completes the proof.  

%
%
%the vertex of the hyperbola $bc=|(d_1d_2- \sigma_1\sigma_2)|$ in the positive quadrant lies within the circle $b^2+c^2 = \sigma_1^2 + \sigma_2^2-d_1^2-d_2^2$.
%
%

%
%From~\eqref{eq:22siga}, if $b$ is further required to be equal to $c$, that is,   
%\begin{equation}
%b = c = \sqrt{\frac{\sigma_1^2 + \sigma_2^2-d_1^2-d_2^2}{2}},
%\end{equation}
%then other than the condition~\eqref{eq:22cond}, the condition  
%\begin{equation}\label{sym2}
% \left.\begin{array}{cc}
% \dfrac{\sigma_1^2 + \sigma_2^2-d_1^2-d_2^2}{2}  = \sigma_1\sigma_2 - d_1d_2, & \mbox{if }
% \sigma_1\sigma_2-d_1d_2 > 0,\\
% \dfrac{\sigma_1^2 + \sigma_2^2-d_1^2-d_2^2}{2} = d_1d_2 - \sigma_1\sigma_2,
%  &  \mbox{if } d_1d_2 - \sigma_1\sigma_2 >0.\end{array}\right.
% \end{equation}
%must also be satisfied. It is easy to see that the condition~\eqref{sym2} is equivalent to the condition~\eqref{sym1}, which completes the proof of Lemma~\ref{lemma:22}.

\end{proof}

From the proof of Lemma~\ref{lemma:21}, it can be easily seen that 
the assumption of $d_1\geq d_2$ for the existence of a $2\times 2$ nonnegative matrix is not a necessity. However, the second inequality in~\eqref{eq:22cond}
must be rewritten as 
\begin{equation}
   \sigma_1 - \sigma_2 \geq |d_1-d_2|.
\end{equation}

We are to provide necessary condition for the existence of such a matrix in the following Lemma.
In particular, necessary conditions
in the cases of symmetric and diagonal matrix are presented, respectively.

\begin{Lemma}\label{lemma:22}
Let $A=
\left[\begin{array}{cc}d_1 & b \\c & d_2\end{array}\right]
\in\mathbb{R}^{n\times n}$
be a nonnegative matrix having singular values $\sigma_1\geq\sigma_2$ and principal diagonal entries $d_1\geq d_2$, renumbering if necessary. 
\begin{itemize}
\item[(i)] If $bc - d_1d_2 \leq 0$, then  
 \begin{subequations}%\label{eq:nisvp}
\begin{eqnarray}
  \sigma_1 + \sigma_2 &\geq& d_1+d_2,\label{ineq:nisvp1}\\
  \sigma_1 - \sigma_2 &\geq& d_1-d_2.\label{ineq:nisvp2}
\end{eqnarray}
\end{subequations}

\item[(ii)] If $bc - d_1d_2 > 0$, then  
\begin{eqnarray}
  \sigma_1 - \sigma_2 &\geq& d_1+d_2.\label{ineq:nisvp3}
\end{eqnarray}
\end{itemize}
Moreover, the equalities of~\eqref{ineq:nisvp1} and~\eqref{ineq:nisvp3}
hold precisely when $A$ is a symmetric matrix and the equality of~\eqref{ineq:nisvp2} holds precisely when $A$ is a diagonal matrix.
\end{Lemma}

\begin{proof}
The necessary condition in~\eqref{ineq:nisvp1} is a well-known result and has been shown in~\cite[Lemma 3]{Thompson1977} followed by the discussion of equality. Hence we omit the proof here.

Because $A$ is a nonnegative matrix  having singular values $\sigma_1\geq\sigma_2$ and principal diagonal entries $d_1\geq d_2$, the numbers $b$ and $c$ must satisfy the relationship given by~\eqref{eq:22sig}. 
Note that the condition that $b, c \in \mathbb{R}$ %are nonnegative numbers 
is equivalent to 
$b^2+c^2-2(d_1d_2-bc) \geq -2 d_1d_2$, and thus by~\eqref{eq:22sig} with 
$bc - d_1d_2 \leq 0$
to 
%\begin{equation}
%d_1^2+d_2^2+b^2+c^2+ 2(d_1d_2-bc) \geq d_1^2+d_2^2 + 2 d_1d_2.
%\end{equation} 
\begin{equation*}
\sigma_1^2+\sigma_2^2 - 2\sigma_1\sigma_2 \geq d_1^2+d_2^2 - 2 d_1d_2, 
\end{equation*} 
i.e., $\sigma_1 -\sigma_2  \geq d_1 - d_2 $. Truly, the equality happens when 
$b = c = 0$, which completes the proof of~\eqref{ineq:nisvp2}. 

On the other hand, if $bc - d_1d_2 > 0$, then the condition that $b, c \in \mathbb{R}$ %are nonnegative numbers 
is equivalent to 
$b^2+c^2-2(bc - d_1d_2) \geq 2 d_1d_2$, and thus by~\eqref{eq:22sig} with 
$bc - d_1d_2 > 0$
to 
%\begin{equation}
%d_1^2+d_2^2+b^2+c^2+ 2(d_1d_2-bc) \geq d_1^2+d_2^2 + 2 d_1d_2.
%\end{equation} 
\begin{equation*}
\sigma_1^2+\sigma_2^2 - 2\sigma_1\sigma_2 \geq d_1^2+d_2^2 + 2 d_1d_2, 
\end{equation*} 
i.e., $\sigma_1 -\sigma_2  \geq d_1 + d_2 $. Also, the equality happens when 
$b = c$, which completes the proof of~\eqref{ineq:nisvp3}.

\end{proof}

For clarity, we state, by combining Lemma \ref{lemma:21} and Lemma \ref{lemma:22},
the sufficient and necessary conditions for the existence of
$2 \times 2$ nonnegative matrices with given singular values and diagonal entries as follows.
\begin{theorem}\label{THM}
There exists a nonnegative matrix 
$
A=
\left[\begin{array}{cc}d_1 & b \\c & d_2\end{array}\right]
\in\mathbb{R}^{n\times n}
$
having  singular values $\sigma_1\geq\sigma_2$ and principal diagonal entries $d_1\geq d_2$, renumbering if necessary,  if and only if   
\begin{subequations}
\begin{eqnarray}
  \sigma_1 + \sigma_2 \geq d_1+d_2,\,\,
  \sigma_1 - \sigma_2 &\geq& d_1-d_2,\quad
     \mbox{while }
 bc - d_1d_2 \leq 0,\\
\sigma_1- \sigma_2 &\geq& d_1 + d_2, \quad  \mbox{while }bc - d_1d_2 > 0. \label{eq:extra}
\end{eqnarray}
\end{subequations} 
% \begin{subequations}\label{eq:nisvp}
%\begin{eqnarray}
%  \sigma_1 + \sigma_2 \geq d_1+d_2,\,\,
%  \sigma_1 - \sigma_2 \geq d_1-d_2,
%    & \quad\mbox{as }
% bc - d_1d_2 \leq 0,\label{eq:22cond}\\
%\sigma_1- \sigma_2 \geq d_1 + d_2, &\quad  \mbox{as }bc - d_1d_2 > 0.
%\label{eq:22cond2}
%\end{eqnarray}
%\end{subequations} 

\end{theorem}

In fact, in \cite{Thompson1977}, the existence of general complexed-valued $n \times n$ matrices
with prescribed singular values and diagonal elements is derived based on 
the results for the $2 \times 2$ case.
However,
the problem of finding sufficient and necessary conditions for the existence of
nonnegative matrices with general sizes remains open.

\section{Projection formulations}\label{sec_proj}
In this section, we wish to apply the concept of ``successive projection" onto two closed sets, for solving the NISVP. The similar strategy has been applied as a very useful paradigm for solving inverse eigenvalue problems with nonnegative entries; see, for example,~\cite{Orsi2006}.

\subsection{General matrices} 
By a general matrix, we mean an $m\times n$ matrix with nonnegative entries and $m\geq n$. To solve general NISVPs with prescribed diagonal elements, let $ \boldsymbol{\sigma} = \{ \sigma_1,\ldots,\sigma_n\}$
and $ \boldsymbol{d} = \{d_1,\ldots,d_n\}$ 
denote the designated sets of singular values and diagonal elements, respectively. Renumbering if necessary, assume both sets are in descending order, i.e., $\sigma_1\geq\ldots\geq\sigma_n$ and $d_1\geq\ldots\geq d_n$.
Let $\mathcal{M}$ be a set of $m\times n$ matrices denoted by 
\begin{equation}
\mathcal{M} = \{A \in\mathbb{R}^{m\times n} | A = U\Sigma V^\top\, \mbox{with  }U\in\mathcal{O}(m), V\in\mathcal{O}(n)\}
\end{equation}
with $\Sigma : = \diag\{\sigma_1,\ldots,\sigma_n\}\in\mathbb{R}^{m\times n}$ and $\mathcal{O}(m)$ the group of all real orthogonal $m\times m $ matrices,
and let $\mathcal{D}_g\\$ be a set of $m\times n$ matrices defined by 
\begin{equation}
\mathcal{D} = \{A \in\mathbb{R}^{m\times n} | A_{ii} = d_i \geq 0\, \mbox{for}\, 1\leq i\leq n\},
\end{equation}
where $A_{ij}$ is the $(i,j)$th element of the matrix $A$.
Also, let $\mathcal{P}$
denote the set of nonnegative $m\times n$ matrices,
\begin{equation}
\mathcal{P} = \{A\in\mathbb{R}^{m\times n} | A_{ij} \geq 0\, \mbox{for all}\, i,j\}.
\end{equation}

It follows that solving the NISVP with prescribed diagonal entries, 
is equivalent to finding an $m\times n$ matrix $X$
which belongs to the intersection of 
$\mathcal{M}$, $\mathcal{D}$ and $\mathcal{P}$, that is, the following problem is considered:
\begin{equation}
\mbox{Find}\, X\in \mathcal{M}\cap\mathcal{D}\cap\mathcal{P}.
\end{equation}

Our criterion is based on alternative projections by 
finding a best projection onto each of the sets $\mathcal{M}$, $\mathcal{D}$ and $\mathcal{P}$. Let's now consider projections onto the different sets given above.  
To begin with, we need to state the following important property, 
the Neumann inequality, which was shown in 1937~\cite{Neumann1937}. Its original proof was particularly complicated. The reader is refereed to~\cite{Mirsky1975} for an alternative but fundamental treatment of the proof. 
For the case of equality condition, its initial proof is given by de S{\'a}~\cite{De1994}. See also~\cite{Rhea2011} for a shorten form of the proof of equality
via the well-known theorem of Eckhart and Young~\cite{Eckart1939}. 

\begin{theorem}\label{ThmNenuman}
Given two matrices $A$ and $B$ in $\mathbb{R}^{m\times n}$, let 
$\rho_1\geq \ldots\geq \rho_n$ and $\sigma_1\geq\ldots\geq \sigma_n$ be singular values of $A$ and $B$, respectively. Then
\begin{equation}
\rm{trace}(A^\top B)\leq \sum_{i=1}^n \rho_i\sigma_i,
\end{equation}
where 
%$\rm{Re}(z)$ denotes the real part of the complex number $z$ and
$\rm{trace}(A)$ denotes the sums of diagonal entries of the matrix $A$.
In particular, the case of equality occurs if and only if there exist orthogonal matrices $U\in\mathcal{R}^{m\times m}$ and $V\in\mathcal{R}^{n\times n}$ such that 
\begin{equation}
A = U\Sigma_A V^\top\quad\mbox{and}\quad B = U\Sigma_B V^\top,
\end{equation}
where $\Sigma_A = \rm{diag}(\rho_1,\ldots,\rho_n) \in\mathbb{R}^{m\times n}$  and $\Sigma_B = \rm{diag}(\sigma_1,\ldots,\sigma_n) \in\mathbb{R}^{m\times n}$~\cite{De1994,Rhea2011}. 
\end{theorem}

Based on the Neumann result, we then 
%Theorem~\ref{thm:Sym} furnishes a convenient criterion for a projection formulation onto a set of symmetric matrices with prescribed singular values. 
have the projection formulation for a general matrix as follows.
\begin{theorem} \label{thm:best}
Let $A$ be a nonnegative $m\times n$ matrix. Suppose that 
\begin{equation}
A = U \Sigma_AV^\top, 
\end{equation}
where 
$U\in\mathbb{R}^{m\times m}$ and $V\in\mathbb{R}^{n\times n}$ are two unitary matrices, and $\Sigma_A = \rm{diag}(\rho_1,\ldots,\rho_n)$ with diagonal elements $\rho_1\ldots\rho_n$, be the singular value decomposition of $A$.
Then $U \Sigma V^{\top}$ is a best approximation in $\mathcal{M}$ to $A$
in the Frobenius norm.
\end{theorem}

\begin{proof}
Note that for all $X\in\mathcal{M}$, 
\begin{equation}\label{XA}
\| X-A\|^{2} = \rm{trace}(\Sigma\Sigma^{\top}) - 2\rm{trace}(XA^{\top}) +\rm{trace}(\tilde{\Sigma}\tilde{\Sigma}^{\top}).
\end{equation}
Thus, finding $X\in\mathcal{M}$ that minimizes $\| X-A\|^{2}$ is equivalent to finding $X\in\mathcal{M}$ that maximizes $\rm{trace}(XA^{\top})$, since the values of $\rm{trace}(\Sigma\Sigma^{\top})$ and $\rm{trace}(\tilde{\Sigma}\tilde{\Sigma}^{\top})$ are constant. 
It then follows from Theorem~\ref{ThmNenuman} that 
if $Y = U \Sigma V^{\top}$, then 
\begin{equation}
\rm{trace}(YA^\top) = \max\limits_{X\in\mathcal{M}} \|X-A\|^2.
\end{equation}
This completes the proof.
\end{proof}

Applying Theorem~\ref{thm:best}, we can now provide a way for the 
projection of the matrix $A$ onto the space $\mathcal{M}$. However, as regards the space $\mathcal{M}$, it is easy to see that $\mathcal{M}$ is nonconvex. Thus, the projection of the matrix $A$ onto the space $\mathcal{M}$ is not unique. 
Thus, to understand the necessary condition of a best projection
would be of interest for further reference. 
For this purpose, 
we shall shortly show the following well-known fact below and embark on relating this result to the proof of the necessary condition of a given matrix $X\in\mathcal{M}$ that minimizes 
$\|X-A\|^2$.

\begin{Lemma}\label{lemmaSkew}
Let $Q$ be an $n\times n$ matrix. Then, $\rm{trace}(Q\Lambda)= 0$
for all skew-symmetric matrices $\Lambda$ if and only if $Q$ is symmetric.
\end{Lemma}
\begin{proof}
Suppose $\rm{trace}(Q\Lambda)= 0$ for all skew-symmetric matrices $\Lambda$ but $Q$ is not symmetric, that is,  assume $Q_{pq}\neq Q_{qp}$ for some entries $(p,q)$
and $(q,p)$ of $Q$. Then choose a matrix $\Lambda$ with $\Lambda_{ij} =  -1$ and $\Lambda_{ji} =1$, if $(i,j) = (p,q)$,  and $\Lambda_{ij} = 0$, otherwise. It follows that 
\begin{equation*}
\rm{trace}(Q\Lambda) = Q_{ij} - Q_{ji} \neq 0,
\end{equation*}
a contradiction.

Conversely, if $Q$ is symmetric, then 
\begin{equation*}
\rm{trace}(Q\Lambda) =\rm{trace}((Q\Lambda)^\top) = 
-\rm{trace}(\Lambda Q)
= -\rm{trace}(Q\Lambda).
\end{equation*}
Thus, $\rm{trace}(Q\Lambda) = 0$, which completes the proof.
\end{proof}

Lemma~\ref{lemmaSkew} then brings out an important point about the discussion of a necessary condition of a best projection onto the set $\mathcal{M}$.  From~\eqref{XA}, we can see that
finding $X\in\mathcal{M}$ that minimizes $\| X-A\|^{2}$ is equivalent to finding $X\in\mathcal{M}$ that maximizes $\rm{trace}(XA^{\top})$. Consider the function
\begin{equation}\label{eq:xa}
f:\mathcal{M} \rightarrow \mathbb{R}, \quad X\rightarrow \rm{trace}(XA^\top).
\end{equation}
and let $T_{X} \mathcal{M}$ denote the tangent space of $\mathcal{M}$ at $X$ given by 
\begin{eqnarray*}
T_{X} \mathcal{M} = \{ XK - HX |  K \in \mathbb{R}^{n \times n} \mbox{ and } H \in \mathbb{R}^{m \times m}
\mbox{ are skew-symmetric}. \}.
\end{eqnarray*}
See, for example,~\cite{Chu1992} for more details.
Then, the derivative of $f$ at a point $X$ in the tangent direction
$XK - HX$ is 
\begin{eqnarray}\label{eq:dxa}
Df(X) (XK - HX) = \rm{trace}((XK - HX)A^{\top}).
\end{eqnarray}
If $X\in \mathcal{M}$ minimizes~\eqref{eq:xa}, then the right hand side 
of~\eqref{eq:dxa} must be equal to zero for all skew-symmetric matrices 
$K\in \mathbb{R}^{n \times n}$ and $H\in \mathbb{R}^{m \times m}$. Upon using Lemma~\ref{lemmaSkew}, we obtain that $A^\top X$ and $XA^\top$ are symmetric. That is, given a matrix $A\in\mathbb{R}^{m\times n}$, a necessary condition of a given matrix $X\in\mathcal{M}$ that minimizes $\|X-A\|^2$ is as follows. 
\begin{remark}\label{RmkNES}
Let $A\in\mathbb{R}^{m\times n}$. If $X\in\mathcal{M}$ minimizes $\|X-A\|^2$, then matrices $A^\top X$ and $XA^\top$ must be symmetric.
\end{remark}

In contrast to $\mathcal{M}$, sets $\mathcal{P}$ and $\mathcal{D}$ are convex. Therefore, projections onto $\mathcal{P}$ and $\mathcal{D}$ are 
unique. The proofs are straightforward and simply based on the direct subtraction between elements of different sets.  

\begin{theorem}\label{ThmPositive}
Given $A\in\mathbb{R}^{m\times n} $, let $A_+$
be a matrix defined by 
\begin{equation*}
(A_+)_{ij} = \left\{\begin{array}{rc} A_{ij}, & \mbox{if}\, A_{ij}\geq 0, \\
0,
 & \mbox{if}\, A_{ij}< 0.\end{array}\right.
\end{equation*}
Then $A_+$ is the best approximation in $\mathcal{P}$ to $A$ in the Frobenius norm.
\end{theorem}
\begin{proof}
Since for each $X\in\mathcal{P}$, we have 
\begin{equation}\label{PXA}
\| X-A\|^{2} = \sum\limits_{i,j}|X_{ij}-A_{ij}|^2, \quad 1\leq i \leq m,\, 1\leq j\leq n.
\end{equation}
To minimize the difference $\| X-A\|^{2}$ is equivalent to minimize  
each entry $|X_{ij}-A_{ij}|$, that is, take 
\begin{equation*}
X_{ij} = \left\{\begin{array}{rc} A_{ij}, & \mbox{if}\, A_{ij}\geq 0, \\
0,
 & \mbox{if}\, A_{ij}< 0,\end{array}\right.
\end{equation*} 
 as desired.
\end{proof}

\begin{theorem}\label{ThmDiagonal}
Given $A\in\mathbb{R}^{m\times n}$, let $A_d$
be a matrix defined by 
\begin{equation*}
(A_d)_{ij} = \left\{\begin{array}{rc} A_{ij}, & \mbox{if}\, i\neq j, \\
d_i,
 & \mbox{if}\, i=j.\end{array}\right.
\end{equation*}
Then $A_d$ is the best approximation in $\mathcal{D}$ to $A$ in the Frobenius norm.
\end{theorem}

\begin{proof}
Finding $X\in\mathcal{D}$ to best approximate $A$ is equivalent to minimizing
\begin{equation}
 \|X-A\|^2 = \sum_{i\neq j} |X_{ij}-A_{ij}|^2+ \sum_{i} |d_{i}-A_{ii}|^2.
 \end{equation}
Since the distance between any matrix in $\mathcal{D}$ and $A$ is at least $(\sum_{i} |d_{i}-A_{ii}|^2)^{\frac{1}{2}}$, the above equality shows that $A_d$ 
is the closest element in $\mathcal{D}$ to $A$.
\end{proof}

From Theorems~\ref{thm:best}, \ref{ThmDiagonal} and \ref{ThmPositive}, we are now in a position to summarize the algorithm
of NISVP with prescribed diagonal entries in Algorithm~\ref{alg_38}.
\setlength{\algomargin}{1.5 ex}
\restylealgo{algoruled}
\SetKwComment{Comment}{}{}
\begin{algorithm}[t!!!!!] \linesnumbered
\caption{Nonnegative Inverse singular value problem with prescribed diagonal entries \hfill $[A_{new}] = \textsc{NISVP}( \boldsymbol{\sigma},  \boldsymbol{d})$}
\KwIn{
$ \boldsymbol{\sigma} = (\sigma_1,\ldots,\sigma_n)$, where $\sigma_1\geq\ldots\geq\sigma_n$,  and 
$ \boldsymbol{d} = (d_1,\ldots,d_n)$
}
\KwOut{An $m\times n$ nonnegative matrix $A$ with desired singular values $ \boldsymbol{\sigma}$
and prescribed diagonal entries $ \boldsymbol{d}$.
}
\Begin{
Randomly generate an initial $m\times n$ nonnegative matrix $A_{new}$\BlankLine
\Repeat
{
$\frac{\|A_{old} - A_{{new}}\|}{\|A_{old}\|} < \epsilon $
}{
\BlankLine
    Calculate the SVD of $A_{new}$ such that 
    \[ A_{new}  =  U\rm{diag}(\delta_1,\ldots,\delta_n)V^\top\]
for $U\in\mathcal{O}(m)$, $V\in\mathcal{O}(n)$ and  
%\begin{equation*}
       %with  
       $\delta_1\geq\ldots\geq\delta_n$\;
 %     \end{equation*}
%where $\delta_1\geq\ldots\geq\delta_n$\;
$A_{new} \leftarrow U\rm{diag}(\sigma_1,\ldots,\sigma_n)V^\top $\;
$A_{old} \leftarrow A_{new}$\;
$A_{new}(find(A_{new}<0)) = 0$\;
$(A_{new})_{ii} \leftarrow d_i $ for all $i$\; \label{seki}

}\label{alg_38}}
\end{algorithm}
\subsection{Convergence analysis}
Throughout this work we examine a successive projection method for solving the NISVP. Since our method use iterative techniques involving sequences, 
this section concludes with a discussion that describes the condition at which convergence occurs. In general, we would like the technique to converge as the solution of NISVP exists and the initial value is chosen close enough to the solution.  

To elucidate our discussion, we apply $\hat{x} = P_C(x)$ to indicate that $\hat{x}$ is a projection of $x$ onto the set $C$.  If $C_1$ and $C_2$ are two closed sets, then for alternating projections between two sets, the distance between two successive projections satisfies the following condition~\cite{Orsi2006}. 

\begin{theorem}\label{thm:Oris}
Given two nonempty closed sets $C_1$ and $C_2$ in a finite dimensional Hilbert space $H$, let $x_0\in C_1$ and let
\begin{equation*}
\hat{x}_i= P_{C_2}(x_i), \quad x_{i+1} = P_{C_1}(\hat{x}_i)
\end{equation*}
be two successive projections between two closed sets, 
then 
\begin{equation*}
\|\hat{x}_{i+1}-x_{i+1}\|\leq \|x_{i+1}-\hat{x}_i\| \leq \|\hat{x}_i-x_i\|,\quad i=1,\ldots.
\end{equation*}
\end{theorem}

\begin{theorem}~\cite[Proposition 2.1]{Helmke1994}\label{ThmHelmke}
The set $\mathcal{M}$ of all real $m\times n$ matrices having singular values $\sigma_1,\ldots,\sigma_r$ occurring with multiplicities $(n_1,\ldots,n_r)$ as 
\begin{equation*}
\mathcal{M} = \{U\Sigma V^\top \in\mathbb{R}^{m\times n} | U\in\mathcal{O}(m), V\in\mathcal{O}(n)\},
\end{equation*}
is a smooth, compact manifold
 of dimension
\begin{equation*}
\dim(\mathcal{M}) = \left\{\begin{array}{cc} n(m-1)-\sum_{i=1}^r \dfrac{n_i(n_i-1)}{2}
 & \mbox{if } \sigma_r > 0 \\n(m-1)-\sum_{i=1}^r \dfrac{n_i(n_i-1)}{2}-n_r(\dfrac{n_r-1}{2} + (m-n)) & \mbox{if } \sigma_r = 0.\end{array}\right.
\end{equation*}

\end{theorem}

Of course, the set $\mathcal{M}_g$ is compact as a direct consequence of Theorem~\ref{ThmHelmke}. 
Also, it is true that the set $\mathcal{D}_{g}\cap\mathcal{P}_{g}$
%(respectively, $\mathcal{D}_{s}\cap\mathcal{P}_{s}$) 
is closed. Following from Theorem~\ref{thm:Oris} and~\ref{ThmHelmke},  
the distance between successive projections
onto $\mathcal{M}_g$ %(respectively, $\mathcal{M}_s$) 
and $\mathcal{D}_{g}\cap\mathcal{P}_{g}$
%(respectively, $\mathcal{D}_{s}\cap\mathcal{P}_{s}$) 
is nonincreasing. In such a case, there is no guarantee that the iteration in Algorithm~\ref{alg_38} %(respectively, Algorithm~\ref{alg_37}) 
will terminate.  It is natural then to discuss the following convergence property.
\begin{theorem}\label{ThmCov}
Suppose that $x_i$'s and $\hat{x}_i$'s are successive projections between two closed sets $C_1$ and $C_2$, that is,
\begin{equation*}
\hat{x}_i= P_{C_2}(x_i), \quad x_{i+1} = P_{C_1}(\hat{x}_i).
\end{equation*}
If the set $C_1$ is compact and $P_{C_2}$ is a continuous operation, then 
\begin{equation}\label{EqCov}
\|x - \hat{x}\| = \lim\limits_{i\rightarrow \infty}\|x_i - \hat{x}_i\|
\end{equation}
for some $x\in {C_1}$ and its projection $\tilde{x}\in C_2$. 

%
%
%Suppose that $\{x_n\}$ is an ordered sequence of projections in 
%the space $\mathcal{M}$ (respectively, $\mathcal{M}_s$) by Algorithm~\ref{alg_38}(respectively, Algorithm~\ref{alg_37}). 
%Let $\{\hat{x}_n\}$ be the corresponding projections of $\{x_n\}$, respectively, onto the space $\mathcal{D}\cap\mathcal{P}$ (respectively, $\mathcal{D}_s\cap\mathcal{P}_s$). Then
%\begin{equation}\label{EqCov}
%\|x - \hat{x}\| = \lim\limits_{n\rightarrow \infty}\|x_i - \hat{x}_i\|
%\end{equation}
%for some $x\in\mathcal{M}$ (respectively, $x\in\mathcal{M}_s$) and its projection $\tilde{x}\in\mathcal{D}\cap\mathcal{P}$ (respectively, $\hat{x}\in\mathcal{D}_s\cap\mathcal{P}_s$). 
\end{theorem}

\begin{proof}
Since $C_1$ is compact and $x_i \in C_1$, 
$\{x_i\}$ has a convergent subsequence, say  $\{x_{i_k}\}$.
Denote that 
\begin{equation}\label{EqOri}
\lim\limits_{k\rightarrow \infty} x_{i_k} = x
\end{equation}
for some $x\in C_1$.
Thus, the fact that the projection $P_{C_2}$ is a continuous operation implies that
$
\lim\limits_{k\rightarrow \infty} P_{C_2}(x_{i_k}) = P_{C_2}(x), 
$
that is, 
\begin{equation}\label{Eqhat}
\lim\limits_{k\rightarrow \infty} \hat{x}_{i_k} = \hat{x}, 
\end{equation}
where $\hat{x}= P_{C_2}(x)\in{C}_2$.
By Theorem~\ref{thm:Oris}, $\{\|{x}_i-\hat{x}_i\|\}$ is a nonincreasing sequence and, hence, converges. It then follows from~\eqref{EqOri} and~\eqref{Eqhat} that 
\begin{equation*}
\lim\limits_{i\rightarrow \infty} \|x_{i}-\hat{x}_i\|  = \lim\limits_{k\rightarrow \infty} \|x_{i_k}-\hat{x}_{i_k}\| = \|x-\hat{x}\|. 
\end{equation*}
 
%The last result follows from

\end{proof}

In the discussion leading up to Theorem~\ref{ThmCov}, we then have the following Corollary, which is similar to the discussion of the NIEP in~\cite[Theorem 5.6]{Orsi2006}, for the NISVP.
\begin{corollary}\label{CorORSI}
Suppose that $x_i$'s is a sequence of successive projections onto the compact $\mathcal{M}_g$. 
%\rm{(}respectively, $\mathcal{M}_s$\rm{)}. 
Then, there is a limit point $x$ of the sequence $\{x_i\}$ satisfies 
\begin{equation*}
\|x - P_{\mathcal{D}_{g}\cap\mathcal{P}_{g}}(x)\| = \lim\limits_{i\rightarrow \infty}\|x_i - P_{\mathcal{D}_{g}\cap\mathcal{P}_{g}}(x_i)\|.
\end{equation*}
%\rm{(}respecitvely, 
%$
%\|x - P_{\mathcal{D}_{s}\cap\mathcal{P}_{s}}(x)\| = \lim\limits_{i\rightarrow \infty}\|x_i - P_{\mathcal{D}_{s}\cap\mathcal{P}_{s}}(x_i)\|
%$\rm{)}. 
\end{corollary}

Corollary~\ref{CorORSI} has a very pleasant interpretation in terms of the analysis of convergence, i.e., if the limit is zero, then every limit point of a subsequence of $\{x_i\}$ is a solution of the NISVP 
%(respectively, the symmetric NISVP) 
with prescribed singular values and diagonal elements.

In fact, another analogous result of NIEP in~\cite[Theorem 5.7]{Orsi2006} for the result of the NISVP algorithm is also true. Since the proof of the theorem is almost identical, we omit the proof here.
\begin{theorem}
Suppose  
$\mathcal{M}_{g}\cap(\mathcal{D}_{g}\cap\mathcal{P}_{g})$ 
% \rm{(}respectively,  $\mathcal{M}_{s}\cap\mathcal{D}_{s}\cap\mathcal{P}_{s}$ \rm{)}
is not empty. Then there is an open neighborhood of 
$\mathcal{M}_{g}\cap
(
\mathcal{D}_{g}\cap\mathcal{P}_{g})^\circ$ 
%\rm{(}respectively,  $\mathcal{M}_{s}\cap
%(
%\mathcal{D}_{s}\cap\mathcal{P}_{s})^\circ$ \rm{)}
from which Algorithm~\ref{alg_38} 
%\rm{(}respectively, Algorithm~\ref{alg_37}\rm{)} 
converges in a single iteration to a solution of the NISVP. 
%\rm{(}respectively, the symmetric NISVP\rm{)}.  
Here, the notation $^\circ$ denotes the interior of a set.  
\end{theorem}

%%
%% Numerical experiments
%%
\section{Numerical experiments}
In this section, we demonstrate 
by numerical examples how Algorithm~\ref{alg_38} 
%and Algorithm~\ref{alg_37} 
can be applied to construct the solution of NISVP with prescribed diagonal elements and, in general, with prescribed entries. 
All computations here are performed in MATLAB/version~2013a on a desktop with a 3.5 GHZ Intel Core i7 processor and 32GB main memory. For both algorithms, the stopping criterion is 
\begin{equation*}
\frac{\|A_{old} - A_{{new}}\|}{\|A_{old}\|} < 10^{-14}. 
\end{equation*}

\begin{example}
To illustrate the feasibility of our approach against problems of relatively large size, we being with a set of singular values and a set of diagonal elements of different size, say $n = 5, 10, 100$. For each size, we consider $100$ experiments.To demonstrate the robustness of our approach, the test data and initial values are randomly generated with entries from a uniform distribution over the interval $[0,10]$. 

Reported in Table~\ref{tab4.1}
%(respectively, in Tabel~\ref{tab4.2})
 is the minimal number of iterations (minit), the maximal number of iterations (maxit), the average number of iterations (aveit), the minimal time (mint), the maximal time (maxt), the average time (avet) and the success rate (sucr) to find a solution by Algorithm~\ref{alg_38}. 
 %(respectively, Algorithm~\ref{alg_37}).
As can be observed, Algorithm~\ref{alg_38} is a highly efficient and effective method for solving the NISVP. What is even remarkable is that the number of iterations is not monotonically increasing
as the size of matrices increases. This might be due to the fact that 
the degrees of freedom of the problem increase as the size of matrices increases.

\begin{table}[h!!!]\label{tab4.1}
\begin{center}
\begin{tabular}
{|c|c|c|c|c|c|c|c|}
\hline   $n\times n$ \rm{matrices}  & \rm{minit} & \rm{maxit} & \rm{aveit} & \rm{mint} & \rm{maxt} & \rm{avet} & \rm{\%sucr} \\
\hline 5 & 31 & 419&73 & 0.0011 & 0.0155 & 0.0026  & 100 \\
\hline 10 & 115 & 3454 &390 & 0.0073 & 0.2622 & 0.0264  & 100 \\
\hline 20 & 48 & 164 &72 & 0.0098 & 0.0311 & 0.0154  & 100 \\
\hline 100 & 49 & 95&62 & 0.1847 & 0.3560 & 0.2464  & 100 \\\hline 
\end{tabular}
\caption{Comparison of the required minit, maxit, aveit, mint, maxt, avet and sucr for solving the NISVP by Algorithm~\ref{alg_38}.}
\end{center}
\end{table}

%By contrast, in Tabel~\ref{tab4.1}, we see that even with $10000$ steps,
%the results are very bad. This might be due to the fact that  
%rounding errors arisen from the SVD can accumulate and amplify in successive iterations and lead to very inaccurate final calculation. 
%
%\begin{table}[h!!!]\label{tab4.2}
%\begin{center}
%\begin{tabular}
%{|c|c|c|c|c|c|c|c|}
%\hline  $n\times n$ \rm{matrices}  & \rm{minit} & \rm{maxit} & \rm{aveit} & \rm{mint} & \rm{maxt} & \rm{avet} & \rm{\%sucr} \\
%\hline 5 & 20 & *($>$10000)&42 & 8.2242e-04 & 0.0040 & 0.0018  & 31 \\
%\hline 10 & 28 &*($>$10000) &36 & 0.0021 & 0.0031 & 0.0028  & 7 \\
%\hline 20 & *($>$10000) & *&* & * & * & *  & 0 \\
%\hline 100 & *($>$10000) & *&* & * & * & *  & 0 \\\hline 
%\end{tabular}
%\caption{Comparison of the required minit, maxit, aveit, mint, maxt, avet and sucr for solving the NISVP by Algorithm~\ref{alg_37}.}
%\end{center}
%\end{table}
\end{example}

 \begin{example}
In this example, we illustrate the application of our approach to construct a nonnegative matrix with prescribed singular values and a specific structure--that is, the NISVP requires the solution to have certain entries in particular positions. This is the so-called the NISVP with prescribed entries (PENISVP)~\cite{Chu05} and can be defined as follows.

\vskip .1in
\begin{minipage}{0.9\textwidth}
\rm{(}\textbf{PENISVP}\rm{)} 
Given a certain subset $\mathcal{I} = \{(i_t,j_t)\}_{t=1}^\ell$ of double subscripts, a certain set of nonnegative numbers $\mathcal{K}  = \{k_t\}_{t=1}^\ell$, and a set of $n$ nonnegative numbers $\{\sigma_i\}_{i=1}^n$, find a nonnegative $m\times n$ matrix $A$, $m\geq n$, having singular values $\sigma_1,\ldots,\sigma_n$ and $A_{i_t,j_t} = k_t$, for $t = 1,\ldots, \ell$.

\end{minipage}
\vskip .1in

So far as we know, no much work has been done for this problem. For this problem, there are lots of interesting questions worthy of investigation. 
%For instance, the Sing-Thompson Theorem with matrices of nonnegative entries can be viewed as a specific kind of PENISVP with $\mathcal{I} = \{(i_t,i_t)\}_{t=1}^n$. 
Our methodology used here is to first consider a set $\mathcal{L}$ denoted by 
\begin{equation}
\mathcal{L} = \{A \in\mathbb{R}^{m\times n} | A_{i_t,j_t} = k_t\, \mbox{for}\, 1\leq t\leq \ell\},%,
\end{equation}
that is, $\mathcal{L}$
is the associated set of nonnegative matrices with entries
specified by the sets  $\mathcal{I}$ and $\mathcal{K}$.
Generalizing what we observed in Theorem~\ref{ThmDiagonal}, we have the following Theorem for solving PENISVP.
\begin{theorem}\label{ThmPrescribed}
Given $A\in\mathbb{R}^{m\times n}$, let $A_p$
be a matrix defined by 
\begin{equation*}
(A_k)_{i,j}  = \left\{\begin{array}{rc} A_{ij}, & \mbox{if}\, (i, j)\neq (i_t,j_t ), \\
k_t,
 & \mbox{if}\, (i, j) =  (i_t,j_t ).\end{array}\right.
\end{equation*}
Then $A_k$ is the best approximation in $\mathcal{L}$ to $A$ in the Frobenius norm.
\end{theorem}

\begin{proof}
Finding $X\in\mathcal{L}$ to best approximate $A$ is equivalent to minimizing
\begin{equation}\label{EQP}
 \|X-A\|^2 = \sum_{ (i, j)\neq (i_t,j_t )} |X_{ij}-A_{ij}|^2
 + \sum_{ (i, j)= (i_t,j_t )} |X_{ij}-A_{ij}|^2.  
 \end{equation}
Note that for the matrix $A$, 
the distance  $\sum_{ (i, j)= (i_t,j_t )} |X_{ij}-A_{ij}|^2
= \sum_{ (i, j)= (i_t,j_t )} |k_t-A_{ij}|^2
$ is a contant. It follows immediately from~\eqref{EQP} that $A_k$ 
is the closest element in $\mathcal{L}$ to $A$.
\end{proof}

This implies that   
to solve the PENISVP is equivalent to 
\begin{equation}
\mbox{find}\,\, X\in \mathcal{M}_{g}\cap\mathcal{P}_{g}\cap\mathcal{L}.
\end{equation}
In this case, we can easily solve the PENISVP by replacing line 1.8 
in Algorithm~\ref{alg_38} 
with the following statement
\begin{equation}\label{PEN}
(A_{new})_{i_t,j_t } \leftarrow k_t, \quad  \mbox{for all }t.
\end{equation}

The example experimented here is borrowed from~\cite{Chu05}
with a minor adjustment
for demonstrating the capacity of solving non-square and nonnegative matrices with arbitrarily prescribed entries. 
It is to find a $6\times 5$ matrix with the structure
\begin{equation*}
A = \left[\begin{array}{ccccc}a_{11} & 1 & 0 & 1 & 0 \\a_{21} & a_{22} & a_{23} & a_{24} & 1 \\0 & 0 & a_{33} & 1 & 0 \\0 & a_{42} & a_{43} & a_{44} & 1 \\a_{51} & a_{52} & a_{53} & a_{54} & a_{55}
\\a_{61} & a_{62} & a_{63} & a_{64} & a_{65}
\end{array}\right],
\end{equation*}
where $a_{ij}\mbox{'s}\geq 0$ are entries needed to be determined corresponding to (presumedly randomly generated) singular values 
$\{0.2780, 0.9786,0.5334,1.2723,3.3108\}$.
To facilitate our illustration, assume the singular values 
have been arranged in the decreasing order such that 
$\sigma_1 = 3.3108,\sigma_2=1.2723,
\sigma_3 = 0.9786, \sigma_4 = 0.5334$ and $\sigma_5 = 0.2780
$. 
Upon applying Algorithm~\ref{alg_38} with the modification in~\eqref{PEN}, we obtain an example of a nonnegative matrix 
\begin{equation*}
A = \left[\begin{array}{ccccc}
    0.2414 &   1   &     0   & 1    &     0\\
    0.9400  &  1.3034  &  0.6074 &   0.6016 &   1\\
         0    &     0 &   0.8033  &  1    &     0\\
         0  &  0.4300  &  0.3211 &   0.3613  &  1\\
             0.3117  &  0.9784  &  0.9789  &  0.7802 &   0.8089\\
    0.5346  &  0.6585  &  0.1276  &  0.1697  &  0.6513\\
\end{array}\right]
\end{equation*}
having the desired singular values and the specific structure.

 \end{example}

 \begin{example}
In this example, we show that Algorithm~\ref{alg_38} works indiscriminately with a nonnegative symmetric matrix. To demonstrate this situation, suppose that $\{   19.1654,   7.9790,    6.4456,    0.6405,$    $0.3812\}$ are the prescribed singular values
and the desired matrix is taken the form 
 \begin{equation*}
A = \left[\begin{array}{ccccc}1 & a_{12}& a_{13} & a_{14} & a_{15} 
\\a_{12} & 1 & a_{23} & a_{24} & a_{25} 
\\a_{13} & a_{23} &1 & a_{34} & a_{35} 
\\a_{14} & a_{24} & a_{34} & 1 & a_{45} 
\\a_{15} & a_{25} & a_{35} & a_{45} & 1
\end{array}\right],
\end{equation*}

Applying our method to the above requirements, we do obtain an example of a symmetric nonnegative matrix 
\begin{equation*}
A = \left[\begin{array}{ccccc}
     1  &  2.8642  &  5.4181  &  0.9486  &  3.4543\\
    2.8642 &   1&    2.7116 &   2.6081   & 7.4045\\
    5.4181  &  2.7116  &  1  &  5.6720   & 5.9509\\
    0.9486   & 2.6081  &  5.6720&    1   & 6.8615\\
    3.4543  &  7.4045  &  5.9509 &   6.8615   & 1
\end{array}\right]
\end{equation*}
having the desired singular values and diagonal entries.  

 \end{example}

\section{Conclusion}
In this paper we have proposed the sufficient and necessary conditions to solve the NISVP for $2\times 2$ matrices, given its principal diagonal elements. For general nonnegative matrices, we have observed how the very challenging and complicated NISVP with prescribed structure can be handled using the successive projection technique. In particular, the case with arbitrarily prescribed entries has not even been considered in the literature. It is of both theoretical and practical interest to see that using the successive projection technique, we are able to solve the very difficult NISVP simply, effectively and efficiently. In other words, this approach offers a unified and effectual avenue of attack to many other ISPs, which we think deserves attention from practitioners in this field.

On the other hand, we feel that the research area of PENISVP is still quite open. It will be interesting to explore this type of problems with other structures such as stochastic, Toeplitz, Hankel matrices, and so on.
Especially, stochastic matrices with prescribed entries have applications in the study of Markov chains since the specified entries may represent certain probability flows under consideration. In addition, finding sufficient and necessary conditions for the existence of a matrix of size larger than $2$
is also worthy of further investigation.
%is also a subject worthy of further investigation. 

\bibliographystyle{siam}
\bibliography{qiep}
\end{document}